\documentclass[12pt]{article}
\usepackage{amssymb}
\usepackage{amsthm}
\usepackage{amsmath}
\usepackage{graphicx}
\newtheorem{theorem}{Theorem}[section]
\newtheorem{definition}[theorem]{Definition}
\newtheorem{lemma}[theorem]{Lemma}
\newtheorem{proposition}[theorem]{Proposition}

\newtheorem{example}[theorem]{Example}
\newtheorem{corollary}[theorem]{Corollary}
\title{Left residuated operators induced by posets with a unary operation}
\author{Ivan Chajda and Helmut L\"anger}
\date{}
\begin{document}
\footnotetext[1]{Support of the research by \"OAD, project CZ~04/2017, as well as by IGA, project P\v rF~2018~012, and support of the research of the second author by the Austrian Science Fund (FWF), project I~1923-N25 entitled ``New perspectives on residuated posets'', is gratefully acknowledged.}
\maketitle
\begin{abstract}
The concept of operator left residuation has been introduced by the authors in their previous paper \cite{CL1}. Modifications of so-called quantum structures, in particular orthomodular posets, like pseudo-orthomodular, pseudo-Boolean and Boolean posets are investigated here in order to show that they are operator left residuated or even operator residuated. In fact they satisfy more general sufficient conditions for operator residuation assumed for bounded posets equipped with a unary operation. It is shown that these conditions may be also necessary if a generalized version using subsets instead of single elements is considered. The above listed posets can serve as an algebraic semantics for the logic of quantum mechanics in a broad sense. Moreover, our approach shows connections to substructural logics via the considered residuation.
\end{abstract}
 


\section{Introduction}

It was shown by G.~Birkhoff and J.~von~Neumann (\cite{BV}) and, independently, by K.~Husimi (\cite H) that orthomodular lattices can serve as an algebraic semantic of the logic of quantum mechanics. Namely, the class of event-state systems in quantum mechanics is usually identified with the set of projection operators on a Hilbert space $\mathbf H$ and this set is in a bijective correspondence with the set of all closed linear subspaces of $\mathbf H$. However, certain doubts concerning the relevance of this representation arose when it was shown that the class of orthomodular lattices arising in this way does not generate the variety of orthomodular lattices. In other words, there exist orthomodular lattices which do not represent a physical system in the logic of quantum mechanics. The reason is that some equational properties of event-state systems are not fairly reflected by the proposed mathematical abstraction. This was the reason for alternative approaches, see e.g.\ \cite{GLP16} and \cite{GLP}. In particular, an algebraic semantic for the logic of quantum mechanics was found among orthomodular posets and their modifications.

Orthomodular lattices have similar properties as algebraic structures used for substructural logics, see e.g.\ \cite{GJKO}. The authors proved in \cite{CL17a} and \cite{CL17b} that every orthomodular lattice can be converted into a so-called left residuated l-groupoid. They showed in \cite{CL2} that this result can be easily extended to a certain class of bounded lattices with a unary operation which, of course, contains the variety of orthomodular lattices. Hence, the natural question arises if this approach can be extended to ordered sets with a unary operation. For this purpose, so-called residuated operators were introduced in \cite{CL1} and several classes of ordered sets with a unary operation turned out to be operator residuated. A prominent role among these posets play the so-called pseudo-orthomodular posets which are a direct generalization of orthomodular lattices, but serve also as good candidates for an algebraic semantic of the logic of quantum mechanics. And again, there arises the natural question if the posets listed in \cite{CL1} really exhaust all possible cases.

The aim of the present paper is to provide several simple conditions under which a bounded poset with a unary operation can be organized into an operator left residuated poset. Similarly as it was done for lattices in \cite{CL2}, we ask if these conditions are not only sufficient but also necessary. It is shown that if subsets instead of single elements are considered then these generalized conditions characterize the class of posets which can be converted into operator residuated ones.

\section{Adjointness of operators}

Recall from \cite{BV} that an {\em orthomodular lattice} is a bounded lattice $(L,\vee,\wedge,{}',0,1)$ with a unary oparation $'$ which is a complementation and an antitone involution (see e.g.\cite K) satisfying the {\em orthomodular law}
\[
x\leq y\text{ implies }x=y\wedge(x\vee y')
\]
or, equivalently,
\[
x\leq y\text{ implies }y=x\vee(y\wedge x').
\]
A {\em left residuated lattice} (or {\em integral l-groupoid} in the terminology of \cite{GJKO}) is a bounded lattice $(L,\vee,\wedge,\odot,\rightarrow,0,1)$ with two more binary operations $\odot$ and $\rightarrow$ satisfying
\begin{align*}
& x\odot1\approx1\odot x\approx x, \\
& x\odot y\leq z\text{ if and only if }x\leq y\rightarrow z\text{ (the so-called {\em left adjointness})}.
\end{align*}
We put $x':=x\rightarrow0$. If $\odot$ is, moreover, commutative then we call the previous property simply {\em adjointness}. If $\odot$ is associative and monotonous in every variable then it is called a {\em t-norm} (see \cite{GJKO}).

It was shown by the authors in \cite{CL17a} and \cite{CL17b} that taking
\begin{align*}
      x\odot y & :=(x\vee y')\wedge y, \\
x\rightarrow y & :=(y\wedge x)\vee x'
\end{align*}
in an orthomodular lattice $\mathbf L=(L,\vee,\wedge,{}',0,1)$ yields a left residuated lattice $(L,\vee,\wedge,\odot,$ $\rightarrow,0,1)$ where $x'=x\rightarrow0$ coincides with the complementation in $\mathbf L$.

However, as shown in \cite{CL1}, if $(P,\leq,{}',0,1)$ is an orthomodular poset (or even a Boolean poset) then such operations $\odot$ and $\rightarrow$ need not exist. In order to avoid these complications we study bounded ordered sets with a unary operation. We introduced in \cite{CL1} the following notion:

\begin{definition}\label{def1}
An {\em operator left residuated poset} is an ordered seventuple $\mathbf P=(P,\leq,{}',M,R,0,1)$ where $(P,\leq,{}',0,1)$ is a bounded poset with a unary operation and $M$ and $R$ are mappings from $P^2$ to $2^P$ satisfying the following conditions for all $x,y,z\in P$:
\begin{enumerate}
\item[{\rm(i)}] $M(x,1)\approx M(1,x)\approx L(x)$,
\item[{\rm(ii)}] $M(x,y)\subseteq L(z)$ if and only if $L(x)\subseteq R(y,z)$,
\item[{\rm(iii)}] $R(x,0)\approx L(x')$.
\end{enumerate}
Condition {\rm(ii)} is called {\em operator left adjointness}. If $M$ is {\em commutative} then {\rm(ii)} is called {\em operator adjointness} only and $\mathbf P$ is called an {\em operator residuated poset}.
\end{definition}

In \cite{CL1}, the definition contains one more condition which, however, follows from (i) and (ii), see the following lemma.

\begin{lemma}\label{lem5}
Every operator left residuated poset $(P,\leq,{}',M,R,0,1)$ satisfies the following condition for all $x,y\in P$:
\[
R(x,y)=P\text{ if and only if }x\leq y.
\]
\end{lemma}

\begin{proof}
For $x,y\in P$ the following are equivalent:
\begin{align*}
R(x,y) & =P, \\
  L(1) & \subseteq R(x,y), \\
M(1,x) & \subseteq L(y), \\
  L(x) & \subseteq L(y), \\
     x & \leq y.
\end{align*}
\end{proof}

For posets $(P,\leq,{}')$ with a unary operation we define the following two conditions:
\begin{align}
             L(x) & \subseteq L(U(L(U(x,y'),y),y'))\text{ for all }x,y\in P,\label{equ1} \\
L(U(L(x,y),y'),y) & \subseteq L(x)\text{ for all }x,y\in P,\label{equ2}
\end{align}
and the following two mappings from $P^2$ to $2^P$:
\begin{align}
M(x,y) & :=L(U(x,y'),y)\text{ for all }x,y\in P,\label{equ3} \\
R(x,y) & :=L(U(L(y,x),x'))\text{ for all }x,y\in P.\label{equ4}
\end{align}

\begin{lemma}\label{lem1}
Let $(P,\leq,{}')$ be a poset with a unary operation and $M$ and $R$ defined by {\rm(\ref{equ3})} and {\rm(\ref{equ4})}, respectively. Then {\rm(\ref{equ1})} implies $M(x,y)\subseteq L(z)\Rightarrow L(x)\subseteq R(y,z)$ and {\rm(\ref{equ2})} implies $L(x)\subseteq R(y,z)\Rightarrow M(x,y)\subseteq L(z)$.
\end{lemma}

\begin{proof}
Assume $a,b,c\in P$. If (\ref{equ1}) and $M(a,b)\subseteq L(c)$ then
\begin{align*}
L(a) & \subseteq L(U(L(U(a,b'),b),b'))=L(U(L(U(a,b'),b)\cap L(b),b'))= \\
     & =L(U(M(a,b)\cap L(b),b'))\subseteq L(U(L(c)\cap L(b),b'))=L(U(L(c,b),b'))=R(b,c).
\end{align*}
If (\ref{equ2}) and $L(a)\subseteq R(b,c)$ then
\begin{align*}
M(a,b) & =L(U(a,b'),b)=L(U(a)\cap U(b'),b)=L(U(L(a))\cap U(b'),b)\subseteq \\
       & \subseteq L(U(R(b,c))\cap U(b'),b)=L(U(L(U(L(c,b),b')))\cap U(b'),b)= \\
       & =L(U(L(c,b),b')\cap U(b'),b)=L(U(L(c,b),b'),b)\subseteq L(c).
\end{align*}
\end{proof}

\begin{definition}\label{def2}
Recall {\rm(}e.g.\ from {\rm\cite{CR})} that a {\em distributive poset} is a poset $(P,\leq)$ satisfying one of the following equivalent identities:
\begin{align*}
L(U(x,y),z) & \approx L(U(L(x,z),L(y,z))), \\
U(L(x,y),z) & \approx U(L(U(x,z),U(y,z))).
\end{align*}
A {\em poset with complementation} is an ordered quintuple $\mathbf P=(P,\leq,{}',0,1)$ such that $(P,\leq,0,1)$ is a bounded poset and $'$ is a unary operation on $P$ satisfying the following conditions for all $x,y\in P$:
\begin{enumerate}
\item[{\rm(i)}] $L(x,x')\approx\{0\}$ and $U(x,x')\approx\{1\}$,
\item[{\rm(ii)}] $x\leq y$ implies $y'\leq x'$,
\item[{\rm(iii)}] $(x')'\approx x$.
\end{enumerate}
\end{definition}

As mentioned in the introduction, we introduce several kinds of posets with complementation which generalize orthomodular lattices.

The poset $\mathbf P$ with complementation is called a {\em Boolean poset} if $(P,\leq)$ is distributive. Of course, every Boolean algebra is a Boolean poset but there are interesting examples of Boolean posets which are not lattices, see e.g.\ \cite{CL2}. In every case, Boolean posets are orthomodular posets and hence they can be considered as quantum structures.

The poset $\mathbf P$ with complementation is called a {\em pseudo-Boolean poset} if it satisfies one of the following equivalent identities:
\begin{align*}
L(U(x,y),y') & \approx L(x,y'), \\
U(L(x,y),y') & \approx U(x,y').
\end{align*}
Pseudo-Boolean posets are certain generalization of Boolean ones but they are closely connected to the following posets.

The poset $\mathbf P$ with complementation is called a {\em pseudo-orthomodular poset} if it satisfies one of the following equivalent identities:
\begin{align*}
L(U(L(x,y),y'),y) & \approx L(x,y), \\
U(L(U(x,y),y'),y) & \approx U(x,y).
\end{align*}
It is evident that pseudo-orthomodular posets are generalizations of orthomodular lattices. Namely, if a poset $(P,\leq,{}',0,1)$ with complementation is a lattice satisfying these identities then $U(x,y)\approx U(x\vee y)$ and $L(x,y)\approx L(x\wedge y)$. Thus our equalities yield
\begin{align*}
L(((x\wedge y)\vee y')\wedge y) & \approx L((x\wedge y)\vee y',y)\approx L(U((x\wedge y)\vee y'),y)\approx L(U(x\wedge y,y'),y)\approx \\
& \approx L(U(L(x\wedge y),y'),y)\approx L(U(L(x,y),y'),y)\approx L(x,y)\approx L(x\wedge y)
\end{align*}
whence
\[
((x\wedge y)\vee y')\wedge y\approx x\wedge y
\]
which is the orthomodular law. Hence, these posets can serve as an algebraic semantics of the logic of quantum mechanics. The advantage of this concept is that we need not assume $x\leq y$ as in the definition of orthomodular posets.

It is easy to see that every Boolean poset is pseudo-Boolean and every pseudo-Boolean poset is pseudo-orthomodular (cf.\ \cite{CL1}).

An example of a pseudo-orthomodular poset which is neither Boolean nor orthomodular is depicted in Fig.~1. It is not orthomodular because e.g.\ $b\leq c'$, but $b\vee c$ does not exist. It is evident that it is not Boolean because it is not distributive.

\vspace*{3mm}

\[
\setlength{\unitlength}{7mm}
\begin{picture}(18,8)
\put(9,0){\circle*{.3}}
\put(6,2){\circle*{.3}}
\put(8,2){\circle*{.3}}
\put(10,2){\circle*{.3}}
\put(12,2){\circle*{.3}}
\put(6,4){\circle*{.3}}
\put(12,4){\circle*{.3}}
\put(6,6){\circle*{.3}}
\put(8,6){\circle*{.3}}
\put(10,6){\circle*{.3}}
\put(12,6){\circle*{.3}}
\put(9,8){\circle*{.3}}
\put(1,4){\circle*{.3}}
\put(17,4){\circle*{.3}}
\put(9,0){\line(-3,2)3}
\put(9,0){\line(-1,2)1}
\put(9,0){\line(1,2)1}
\put(9,0){\line(3,2)3}
\put(9,8){\line(-3,-2)3}
\put(9,8){\line(-1,-2)1}
\put(9,8){\line(1,-2)1}
\put(9,8){\line(3,-2)3}
\put(6,2){\line(0,1)4}
\put(12,2){\line(0,1)4}
\put(6,4){\line(1,1)2}
\put(6,2){\line(1,1)4}
\put(8,2){\line(1,1)4}
\put(10,2){\line(1,1)2}
\put(8,2){\line(-1,1)2}
\put(10,2){\line(-1,1)4}
\put(12,2){\line(-1,1)4}
\put(12,4){\line(-1,1)2}
\put(1,4){\line(2,-1)8}
\put(1,4){\line(2,1)8}
\put(17,4){\line(-2,-1)8}
\put(17,4){\line(-2,1)8}
\put(8.85,-.75){$0$}
\put(5.4,1.9){$a$}
\put(7.2,1.9){$b$}
\put(10.45,1.9){$c$}
\put(12.4,1.9){$d$}
\put(5.4,3.9){$e$}
\put(.4,3.9){$f$}
\put(12.4,3.9){$e'$}
\put(17.4,3.9){$f'$}
\put(5.3,5.9){$d'$}
\put(7.2,5.9){$c'$}
\put(10.45,5.9){$b'$}
\put(12.4,5.9){$a'$}
\put(8.85,8.4){$1$}
\put(8.2,-1.5){{\rm Fig.\ 1}}
\end{picture}
\]

\vspace*{8mm}

\begin{theorem}\label{th1}
Let $(P,\leq,{}',0,1)$ be a bounded poset with a unary operation satisfying both conditions {\rm(\ref{equ1})} and {\rm(\ref{equ2})} and the identity $1'\approx0$ and $M$ and $R$ defined by {\rm(\ref{equ3})} and {\rm(\ref{equ4})}, respectively. Then $(P,\leq,{}',M,R,0,1)$ is an operator left residuated poset.
\end{theorem}

\begin{proof}
\
\begin{enumerate}
\item[(i)] $M(x,1)\approx L(U(x,1'),1)\approx L(U(x,1'))\approx L(U(x,0))\approx L(U(x))\approx L(x)$, \\
$M(1,x)\approx L(U(1,x'),x)\approx L(1,x)\approx L(x)$,
\item[(ii)] follows from Lemma~\ref{lem1},
\item[(iii)] $R(x,0)\approx L(U(L(0,x),x'))\approx L(U(0,x'))\approx L(U(x'))\approx L(x')$.
\end{enumerate}
\end{proof}

We show that the posets mentioned above are among those assumed in Theorem~\ref{th1}.

\begin{example}\label{ex1}
Every pseudo-orthomodular poset satisfies both {\rm(\ref{equ1})} and {\rm(\ref{equ2})}. This can be seen as follows: If $(P,\leq,{}',0,1)$ is a pseudo-orthomodular poset and $a,b\in P$ then
\begin{align*}
             L(a) & =L(U(a))\subseteq L(U(a,b'))=L(U(L(U(a,b'),b),b')), \\
L(U(L(a,b),b'),b) & =L(a,b)\subseteq L(a).
\end{align*}
\end{example}

For posets $(P,\leq,{}')$ with a unary operation we define the following two conditions:
\begin{align}
        L(x) & \subseteq L(U(L(x,y),y'))\text{ for all }x,y\in P,\label{equ7} \\
L(U(x,y'),y) & \subseteq L(x)\text{ for all }x,y\in P\label{equ8}
\end{align}
and the following two mappings from $P^2$ to $2^P$:
\begin{align}
M(x,y) & :=L(x,y)\text{ for all }x,y\in P,\label{equ9} \\
R(x,y) & :=L(U(y,x'))\text{ for all }x,y\in P.\label{equ10}
\end{align}
Observe that $M$ is commutative.

\begin{lemma}\label{lem3}
Let $(P,\leq,{}')$ be a poset with a unary operation and $M$ and $R$ defined by {\rm(\ref{equ9})} and {\rm(\ref{equ10})}, respectively. Then {\rm(\ref{equ7})} implies $M(x,y)\subseteq L(z)\Rightarrow L(x)\subseteq R(y,z)$ and {\rm(\ref{equ8})} implies $L(x)\subseteq R(y,z)\Rightarrow M(x,y)\subseteq L(z)$.
\end{lemma}

\begin{proof}
Assume $a,b,c\in P$. If (\ref{equ7}) and $M(a,b)\subseteq L(c)$ then
\begin{align*}
L(a) & \subseteq L(U(L(a,b),b'))=L(U(L(a,b))\cap U(b'))=L(U(M(a,b))\cap U(b'))\subseteq \\
     & \subseteq L(U(L(c))\cap U(b'))=L(U(c)\cap U(b'))=L(U(c,b'))=R(b,c).
\end{align*}
If (\ref{equ8}) and $L(a)\subseteq R(b,c)$ then
\[
M(a,b)=L(a,b)=L(a)\cap L(b)\subseteq R(b,c)\cap L(b)=L(U(c,b'))\cap L(b)=L(U(c,b'),b)\subseteq L(c).
\]
\end{proof}

\begin{theorem}\label{th2}
Let $(P,\leq,{}',0,1)$ be a bounded poset with a unary operation satisfying both conditions {\rm(\ref{equ7})} and {\rm(\ref{equ8})} and $M$ and $R$ defined by {\rm(\ref{equ9})} and {\rm(\ref{equ10})}, respectively. Then $(P,\leq,{}',M,R,0,1)$ is an operator residuated poset.
\end{theorem}

\begin{proof}
\
\begin{enumerate}
\item[(i)] $M(x,1)\approx L(x,1)\approx L(x)$, \\
$M(1,x)\approx L(1,x)\approx L(x)$,
\item[(ii)] follows from Lemma~\ref{lem3},
\item[(iii)] $R(x,0)\approx L(U(0,x'))\approx L(U(x'))\approx L(x')$ .
\end{enumerate}
Since $M$ is commutative, $(P,\leq,{}',M,R,0,1)$ is an operator residuated poset.
\end{proof}

Again, pseudo-Boolean and hence also Boolean posets are among those posets assumed in Theorem~\ref{th2}, see the following example.

\begin{example}\label{ex2}
Every pseudo-Boolean poset satisfies {\rm(\ref{equ7})} and {\rm(\ref{equ8})}. This can be seen as follows: If $(P,\leq,{}',0,1)$ is a pseudo-Boolean poset and $a,b\in P$ then
\begin{align*}
        L(a) & =L(U(a))\subseteq L(U(a,b'))=L(U(L(a,b),b')), \\
L(U(a,b'),b) & =L(a,b)\subseteq L(a).
\end{align*}
\end{example}

Combining Theorems~\ref{th1} and \ref{th2} and Examples~\ref{ex1} and \ref{ex2} we conclude

\begin{corollary}\label{cor1}
If $(P,\leq,{}',0,1)$ is a pseudo-Boolean poset and $M$ and $R$ are defined by {\rm(\ref{equ9})} and {\rm(\ref{equ10})}, respectively, then $(P,\leq,{}',M,R,0,1)$ is an operator residuated poset. If $(P,\leq,{}',0,1)$ is a pseudo-orthomodular poset satisfying the identity $1'\approx0$ and $M$ and $R$ are defined by {\rm(\ref{equ3})} and {\rm(\ref{equ4})}, respectively, then $(P,\leq,{}',M,R,0,1)$ is an operator left residuated poset.
\end{corollary}

It is well-known that in a residuated lattice each of the operations $\odot$ and $\rightarrow$ determines the other one. We can prove a similar result also for the posets listed above.

\begin{proposition}
\
\begin{enumerate}
\item[{\rm(i)}] If $(P,\leq,{}',0,1)$ is a pseudo-orthomodular poset and $M$ and $R$ are defined by {\rm(\ref{equ3})} and {\rm(\ref{equ4})}, respectively, then
\begin{align*}
L((M(y',x))') & \approx R(x,y), \\
L((R(y,x'))') & \approx M(x,y). 
\end{align*}
\item[{\rm(ii)}] If $(P,\leq,{}',0,1)$ is a poset with complementation and $M$ and $R$ are defined by {\rm(\ref{equ9})} and {\rm(\ref{equ10})}, respectively, then
\begin{align*}
L((M(y',x))') & \approx R(x,y), \\
L((R(y,x'))') & \approx M(x,y). 
\end{align*}
\end{enumerate}
\end{proposition}

\begin{proof}
\
\begin{enumerate}
\item[(i)]
\begin{align*}
L((M(y',x))') & \approx L((L(U(y',x'),x))')\approx L(U(L(y,x),x'))\approx R(x,y), \\
L((R(y,x'))') & \approx L((L(U(L(x',y),y')))')\approx L(U(L(U(x,y'),y)))\approx L(U(x,y'),y)\approx \\
              & \approx M(x,y),
\end{align*}
\item[(ii)]
\begin{align*}
L((M(y',x))') & \approx L((L(y',x))')\approx L(U(y,x'))\approx R(x,y), \\
L((R(y,x'))') & \approx L((L(U(x',y')))')\approx L(U(L(x,y)))\approx L(x,y)\approx M(x,y).
\end{align*}
\end{enumerate}
\end{proof}

\section{A characterization of posets satisfying generalized operator residuation}

The conditions (\ref{equ7}),(\ref{equ8}) as well as (\ref{equ9}),(\ref{equ10}) which are formulated for variables can be expressed also for subsets of $P$ in the following way:

For posets $(P,\leq,{}')$ with a unary operation we define $A':=\{x'\mid x\in A\}$ for all subsets $A$ of $P$. Moreover, we define the following two conditions:
\begin{align}
        L(A) & \subseteq L(U(L(A,B),B'))\text{ for all }A,B\subseteq P,\label{equ11} \\
L(U(A,B'),B) & \subseteq L(A)\text{ for all }A,B\subseteq P\label{equ12}
\end{align}
and the following two binary operations on $2^P$:
\begin{align}
M(A,B) & :=L(A,B)\text{ for all }A,B\subseteq P,\label{equ13} \\
R(A,B) & :=L(U(B,A'))\text{ for all }A,B\subseteq P.\label{equ14}
\end{align}
In case $A=\{x\}$ and $B=\{y\}$ we will write simply $M(x,y)$ and $R(x,y)$ as previously. Of course, taking singletons in (\ref{equ11}) and (\ref{equ12}) instead of $A$ and $B$ yields (\ref{equ7}) and (\ref{equ8}), respectively. Hence, the new conditions and definitions include the previous ones as a particular case. Also our definition of operator adjointness can be extended to subsets of $P$ as follows:
\begin{align}
& \text{for all }A,B,C\subseteq P, M(A,B)\subseteq L(C)\text{ implies }L(A)\subseteq R(B,C),\label{equ15} \\
& \text{for all }A,B,C\subseteq P, L(A)\subseteq R(B,C)\text{ implies }M(A,B)\subseteq L(C).\label{equ16}
\end{align}
The ordered seventuple $\mathbf P=(P,\leq,{}',M,R,0,1)$ will be called a {\em generalized operator left residuated poset} if it satisfies (i) and (iii) of Definition~\ref{def1} as well as (\ref{equ15}) and (\ref{equ16}), i.e.\ if for all $A,B,C\subseteq P$, 
\begin{align}
& M(A,B)\subseteq L(C)\text{ is equivalent to }L(A)\subseteq R(B,C).\label{equ5}
\end{align}
Condition (\ref{equ5}) will be called {\em generalized operator left adjointness}. If $M$ is commutative then (\ref{equ5}) is called {\em generalized operator adjointness} only and $\mathbf P$ is called a {\em generalized operator residuated poset}. It is evident that taking singletons instead of $A,B,C$ in generalized operator adjointness, we obtain condition (ii) from Definition~\ref{def1}. Now we are able to prove a result analogous to Lemma~\ref{lem3}, but in a stronger version.

\begin{theorem}\label{th3}
Let $(P,\leq,{}')$ be a poset with a unary operation and $M$ and $R$ defined by {\rm(\ref{equ13})} and {\rm(\ref{equ14})}, respectively. Then {\rm(\ref{equ11})} and {\rm(\ref{equ15})} are equivalent, and {\rm(\ref{equ12})} and {\rm(\ref{equ16})} are equivalent.
\end{theorem}

\begin{proof}
Assume $A,B,C\subseteq P$. \\
(\ref{equ11}) $\Rightarrow$ (\ref{equ15}): \\
If $M(A,B)\subseteq L(C)$ then
\begin{align*}
L(A) & \subseteq L(U(L(A,B),B'))=L(U(L(A,B))\cap U(B'))=L(U(M(A,B))\cap U(B'))\subseteq \\
     & \subseteq L(U(L(C))\cap U(B'))\subseteq L(U(C)\cap U(B'))=L(U(C,B'))=R(B,C).
\end{align*}
(\ref{equ15}) $\Rightarrow$ (\ref{equ11}): \\
Any of the following assertions implies the next one:
\begin{align*}
L(A,B) & \subseteq L(A,B), \\
M(A,B) & \subseteq L(A,B), \\
M(A,B) & \subseteq L(L(A,B)), \\
  L(A) & \subseteq R(B,L(A,B)), \\
  L(A) & \subseteq L(U(L(A,B),B')).
\end{align*}
(\ref{equ12}) $\Rightarrow $(\ref{equ16}): \\
If $L(A)\subseteq R(B,C)$ then
\begin{align*}
M(A,B) & =L(A,B)=L(A)\cap L(B)\subseteq R(B,C)\cap L(B)=L(U(C,B'))\cap L(B)= \\
       & =L(U(C,B'),B)\subseteq L(C).
\end{align*}
(\ref{equ16}) $\Rightarrow$ (\ref{equ12}): \\
Any of the following assertions implies the next one:
\begin{align*}
  L(U(A,B')) & \subseteq L(U(A,B')), \\
  L(U(A,B')) & \subseteq R(B,A), \\
M(U(A,B'),B) & \subseteq L(A), \\
L(U(A,B'),B) & \subseteq L(A).
\end{align*}
\end{proof}

By Theorem~\ref{th3} we obtain a stronger version of a result analogous to Theorem~\ref{th2}.

\begin{corollary}\label{cor2}
Let $(P,\leq,{}')$ be a poset with a unary operation and $M$ and $R$ defined by {\rm(\ref{equ13})} and {\rm(\ref{equ14})}, respectively. Then $(P,\leq,{}',M,R,0,1)$ is a generalized operator residuated poset if and only if it satisfies both conditions {\rm(\ref{equ11})} and {\rm(\ref{equ12})}.
\end{corollary}

Authors' addresses:

Ivan Chajda \\
Palack\'y University Olomouc \\
Faculty of Science \\
Department of Algebra and Geometry \\
17.\ listopadu 12 \\
771 46 Olomouc \\
Czech Republic \\
ivan.chajda@upol.cz

Helmut L\"anger \\
TU Wien \\
Faculty of Mathematics and Geoinformation \\
Institute of Discrete Mathematics and Geometry \\
Wiedner Hauptstra\ss e 8-10 \\
1040 Vienna \\
Austria, and \\
Palack\'y University Olomouc \\
Faculty of Science \\
Department of Algebra and Geometry \\
17.\ listopadu 12 \\
771 46 Olomouc \\
Czech Republic \\
helmut.laenger@tuwien.ac.at
\end{document}